\RequirePackage[english]{babel}
\documentclass{amsart}

\usepackage{amssymb,amscd,amsthm}

\title[On the homotopy type of definable groups]{On the homotopy type of definable groups in an o-minimal structure}

\author[A.~Berarducci]{Alessandro Berarducci}
\address{Universit\`a di Pisa, Dipartimento di Matematica, Largo Bruno Pontecorvo 5, 56127 Pisa, Italy}
\email{berardu@dm.unipi.it}
\thanks{Partially supported by PRIN 2007PRYAAF\_004: O-minimalit\`a - Metodi e modelli non standard - Teoria degli insiemi}
\author[M.~Mamino]{Marcello Mamino}
\address{Classe di Scienze - Scuola Normale Superiore, Piazza dei Cavalieri, 7, 56126 Pisa, Italy}
\email{m.mamino@sns.it}
\date{19 Nov. Nov. 2009}
\subjclass[2000]{03C64, 03H05, 22E15}
\keywords{Homotopy, Definable groups, o-minimality}

\DeclareMathOperator{\R}{\mathbb R}
\DeclareMathOperator{\NN}{\mathbb N}

\DeclareMathOperator{\Z}{\mathbb Z}

\DeclareMathOperator{\dom}{dom} 
\DeclareMathOperator{\img}{Im}
\DeclareMathOperator{\Ad}{Ad}
\DeclareMathOperator{\Aut}{Aut}
\DeclareMathOperator{\Int}{Int}

\newcommand{\ov}{\overline}

\newcommand{\uc}{\widetilde}

\theoremstyle{plain}
\newtheorem{theorem}{Theorem}
\newtheorem{lemma}[theorem]{Lemma}

\newtheorem{proposition}[theorem]{Proposition}
\newtheorem{corollary}[theorem]{Corollary}

\newtheorem{claim}{Claim}

\newtheorem{fact}[theorem]{Fact}
\theoremstyle{definition}
\newtheorem{remark}[theorem]{Remark}
\newtheorem{definition}[theorem]{Definition}
\newtheorem{example}[theorem]{Example}
\newtheorem{exercise}[theorem]{Exercise}

\newtheorem{assumption}[theorem]{Assumption}

\numberwithin{theorem}{section}

\newcommand{\bt}{\begin{theorem}}
\newcommand{\et}{\end{theorem}}

\newcommand{\bl}{\begin{lemma}}
\newcommand{\el}{\end{lemma}}

\newcommand{\bfa}{\begin{fact}}
\newcommand{\efa}{\end{fact}}

\newcommand{\bexa}{\begin{example}}
\newcommand{\eexa}{\end{example}}

\newcommand{\bexe}{\begin{exercise}}
\newcommand{\eexe}{\end{exercise}}

\newcommand{\bprop}{\begin{proposition}}
\newcommand{\eprop}{\end{proposition}}

\newcommand{\bp}{\begin{proof}}
\newcommand{\ep}{\end{proof}}

\newcommand{\bc}{\begin{corollary}}
\newcommand{\ec}{\end{corollary}}

\newcommand{\bd}{\begin{definition}}
\newcommand{\ed}{\end{definition}}

\newcommand{\br}{\begin{remark}}
\newcommand{\er}{\end{remark}}

\newcommand{\w}[1]{\widetilde {#1}} 
\newcommand{\ca}[1]{{\mathcal #1}}

\newcommand{\lie}[1]{#1/{#1}^{00}}
\newcommand{\fund}[2]{\pi(#1,#2)}
\newcommand{\dfundg}[1]{\pi_1^{\rm{def}}(#1)}
\newcommand{\dfund}[2]{\pi^{\rm{def}}(#1,#2)}
\newcommand{\fundg}[1]{\pi_1(#1)}
\newcommand{\rest}[1]{\hspace {-0.2em} \upharpoonright_{#1}}

\newcommand{\pushout}[8]{
\begin{CD}
#1 @>#5>> #2\\
@VV#6V @VV#7V\\
#3 @>#8>> #4
\end{CD}
}

\newenvironment{acknowledgements}{{\bf Acknowledgements.}}{}
	
\def\MM#1{\raise0.7ex\hbox{\tiny\ding{170}}\marginpar{$\mathcal{M}^2$: {\footnotesize #1}}}

\begin{document}

\begin{abstract}
We consider definably compact groups in an o-minimal expansion of a real closed field. It is known that to each such group $G$ is associated a natural exact sequence $1 \to G^{00} \to G \to \lie{G}\to 1$ where $G^{00}$ is the ``infinitesimal subgroup'' of $G$ and $\lie G$ is a compact real Lie group. We show that given a connected open subset $U$ of $\lie G$ there is a canonical isomorphism between the fundamental group of $U$ and the o-minimal fundamental group of its preimage under the projection $\tau \colon G \to \lie G$. We apply this result to show that there is a natural exact sequence $1 \to G^{00} \to \uc{G} \to \uc{\lie G} \to 1$ where $\uc{G}$ is the (o-minimal) universal cover of $G$, and $\uc{\lie G}$ is the universal cover of the real Lie group $\lie G$. 
Finally we prove that the (Lie-)isomorphism type of $\lie G$ determines the definable homotopy type of $G$. In the semisimple case a stronger result holds: $\lie G$ determines $G$ up to definable isomorphism. 
Our results depend on the study of the o-minimal fundamental groupoid of $G$ and the homotopy properties of the projection $G \to \lie G$.  
\end{abstract} 

\maketitle

\section{Introduction} 
Definable groups in o-minimal expansions of a real closed field have been studied by several authors (see \cite{Ot:08} for a survey). The class of such groups includes all semialgebraic groups over a real closed field, which in turn includes all algebraic groups over an algebraically closed field of characteristic zero (with a fixed maximal real closed subfield). Starting
with \cite{Pi:88}, the main line of research on definable groups has been guided by the
analogy with real Lie groups. However there are also some striking differences: the
correspondence between Lie groups and Lie algebras works well in the
simple and semisimple case (see \cite{PePiSt:00}), but fails in the abelian case due to the possible absence of one-parameter subgroups (see \cite{St:94,PeSt:99}). To remedy this, there have been two lines of attack in the study of definable groups. One through the study of generic subsets, a kind of
substitute for the Haar measure (see \cite{Ke:87,BeOt:04,PePi:07,HrPePi:08,HrPi:07}). The other through the study of the o-minimal Euler
characteristic (\cite{St:94}) and of the definable homotopy invariants of a definable group (see \cite{BeOt:02,EdOt:04}). The two
lines of research are highly intertwined and advances in each side have been possible
through the advances on the other side. By taking a quotient by the ``infinitesimal
subgroup'' one can associate in a canonical way to every definably compact group $G$ a compact
real Lie group $\lie G$ (\cite{Pi:04,BeOtPePi:05}), giving rise to a well behaved functor $F\colon G \to \lie G$ from definably compact groups to compact real Lie groups. A combination of the above mentioned approaches has lead to the determination of the dimension of the associated Lie group (\cite{HrPePi:08}), to the proof of ``the compact domination conjecture'' in \cite{HrPi:07} (extended in \cite{HrPePi:08b} to the non-abelian case), and to various comparison theorems
between the homotopy invariants of a definable group and those of the associated
Lie group (\cite{Be:07,Be:08,BeMaOt:09}). 

An important tool in these investigations has been the study of the definable fundamental group $\dfundg G$. If $G$ is definably compact abelian and definably connected, $\dfundg G \cong \Z^n$ where $n= \dim G$ (\cite{EdOt:04}).
In general if $G$ is definably compact and definably connected, $\dfundg G \cong \fundg {\lie G}$ (\cite{BeMaOt:09}). We will establish a local version of this result, namely we show that for each open connected subset $U$ of $\lie G$ there is an isomorphism $\dfundg {\tau^{-1}(U)} \cong \fundg U$ where $\tau \colon G \to \lie G$ is the projection. To prove this, it is convenient to first establish a similar result for the ``fundamental groupoid'', namely we show that, for each open subset $U$ of $\lie G$, there is a natural homomorphisms of groupoids $\dfund {\tau^{-1}(U)}{\tau^{-1}(U)} \to \fund U U$ (Theorem \ref{top-lemma}). Note that there is no obvious way to define the required homomorphisms. 
The difficulty is the following. If $\gamma$ is a definable path in $G$ and $\tau\colon G\to \lie G$ is the natural map, then $\tau \circ \gamma$ is {\em not} a path in $\lie G$ because we are working in different categories: definable paths in $G$ are parametrized by intervals of the o-minimal structure, while paths in $\lie G$ are parametrized by intervals in $\R$. We will however show that $\tau \circ \gamma$ can be approximated, in a suitable sense, by a path in $\lie G$ (see Definition \ref{defi}). The process of approximation depends on the choice of suitable coverings of $G$ and $\lie G$, but the resulting homomorphisms are independent of the coverings and enjoy good functorial properties. Moreover (taking $U=\lie G$) the groupoid homomorphism $\dfund G G \to \dfund {\lie G}{\lie G}$ respects to the action of $G$. We use these facts to show that there is a natural exact sequence $1 \to G^{00} \to \uc{G} \to \uc{\lie G} \to 1$ where $\uc{G}$ is the (o-minimal) universal cover of $G$, and $\uc{\lie G}$ is the universal cover of the real Lie group $\lie G$. 

In the last part of the paper we apply the results on the fundamental groupoid to try to understand up to which extent $\lie G$ determines $G$, where $G$ is definably compact and definably connected. In \cite{HrPePi:08b} it is proved that in the group language $G$ is elementary equivalent to $\lie G$. In the same paper it is proved that the commutator subgroup $[G,G]$ is definable (and semisimple) and $G$ is the almost direct product of $[G,G]$ and the identity component $Z^0(G)$ of its center (an abelian definably connected group). 
The study of definably compact definably connected groups can thus be reduced to a large extent to the abelian case and the semisimple case. So let us first consider these two cases separately.

The study of semisimple definable groups can be essentially reduced to the study of groups defined in the real field $(\R,<,+,\cdot)$. This depends on the fact that any o-minimal expansion of a field contains an isomorphic copy of the field $\R^{\rm{alg}}$ of the real algebraic numbers, and any definably connected semisimple definable group is definably isomorphic to a semialgebraic group defined over $\R^{\rm{alg}}$ (\cite{EdJoPe:07} or \cite[Theorem 4.4]{HrPePi:08b}). Using this fact we show that any Lie isomorphism from $\lie G$ to $\lie {G'}$ induces, if the o-minimal structure is sufficiently saturated, a definable isomorphism from $G$ to $G'$ (see Theorem \ref{semisimple-case} for the full statement).

The abelian case is in this respect more complicated. Recall that any compact connected abelian real Lie group of dimension $n$ is Lie isomorphic to the $n$-dimensional torus. The corresponding result fails for definable groups due to the possible lack of definable one-dimensional subgroups.  However in \cite{BeMaOt:09} it is proved that any two definably compact definably connected abelian groups $G$ and $G'$ of the same dimension are definably homotopy equivalent. Using the work on the fundamental groupoid we strengthen this result as follows. Given a finite subgroup $\Gamma$ of $G$ there is a definable homotopy equivalence $f:G \to G'$ that restricted to $\Gamma$ is a group isomorphism onto its image $\Gamma'$ and moreover $f(cx) = f(c)f(x)$ for all $c$ is in $\Gamma$ and all $x$ in $G$ (for the full statement see Theorem \ref{abelian-case}). 

Combining the results on the abelian case and the semisimple case, we obtain that given two definably compact definably connected groups $G$ and $G'$ with $\lie G \cong \lie {G'}$, then $G$ and $G'$ are definably homotopy equivalent \footnote{This also follows by the results of Elias Baro in \cite[\S4]{Ba:09} obtained independently at the same time and by different methods (private communication).}. Moreover, given a finite subgroup $\Gamma$ of $G$, we can choose the homotopy equivalence to be ``$\Gamma$-equivariant'', namely to respect the action of $\Gamma$ (for the full statement see Theorem \ref{full}). In general $\Gamma$-equivariant maps are a useful tool to deal with almost direct products (see Lemma \ref{glue-lemma}). 

Having proved that the homotopy type of a definably compact group $G$ is determined by $\lie G$, let us observe that if $G$ is not definably compact, the study of its homotopy type can be reduced to the compact case by the results contained in \cite{Co:09}. 

Our proofs make essential use of the fact that $G^{00}$ is a decreasing intersection of a countable family of definably simply connected (actually even definably contractible) subsets of $G$. This was established in \cite{Be:08} making use of the ``compact domination conjecture'' proved in \cite{HrPePi:08b} (extended in \cite{HrPePi:08b} to the non-abelian case). To a large extent our analysis remains valid replacing $\tau \colon G \to \lie G$ with any other ``logical quotient'' $f\colon X\to Y$ (as in \cite[\S 2]{Pi:04}) where $X$ is a definable set and $Y$ is a locally simply connected compact space, provided that each fiber $f^{-1}(y)$ of $f$ is a a decreasing intersection $\bigcap_{n\in \NN} C_i$ of a countable family of definably simply connected sets. Under these assumptions Theorem \ref{top-lemma} gives us an isomorphism $\dfundg X \cong \fundg Y$. In the group situation, namely when $f$ is the map $\tau \colon G \to \lie G$, we obtain stronger results (e.g. Theorem \ref{full}). Essentially this depends on the fact (established in \cite{BeMaOt:09}) that all the higher homotopy groups of a definably compact abelian group are zero, so the fundamental group already gives us all the relevant homotopic information.

\begin{acknowledgements} We thanks Elias Baro for his comments on a preliminary draft of this paper.  The main results of this paper were presented at the Complexo Interdisciplinar of the University of Lisbon in the period 10-16 March 2009. The first author thanks Tamara Servi, Mario Edmundo, Fernando Ferreira, and all the participants to the seminars for their invitation and warm welcome.
\end{acknowledgements} 

\section{Topological notions}\label{def-spaces}
Fix an o-minimal structure $M$ expanding a field. By ``definable'' we mean ``definable in $M$''. As usual we put on $M$ the order topology and on $M^n$ the product topology. Unless otherwise stated subsets of $M^n$ will be endowed with the induced topology. 
Let $X$ be a subset, not necessarily definable, of a definable set. We say that $X$ is {\bf definably connected} if it cannot be partitioned in two non-empty open subsets which are relatively definable in $X$, where by definition $A\subset X$ is relatively definable if it is the intersection of a definable set with
$X$. We say that $X$ is {\bf definably path-connected} if each pair of points
of $X$ can be connected by a definable (continuous) path in $X$. Finally $X$ is {\bf definably simply connected} if it is definably path connected and any two definable paths in $X$ with the same endpoints are definably homotopic relative to the endpoints. 

Dropping the definability conditions one obtains the corresponding classical notions. The reader should be however be warned that definably connected spaces are in general not connected in the classical sense (unless the o-minimal structure is based on the reals), and similarly for the other notions. 
Let us also recall that a topological space $X$ is {\bf locally simply connected} if for each open set $O$ in $X$ and each $x\in O$, there is an open neighbourhood $V\subset O$ of $x$ which is simply connected. There are several possible definable versions of this notion, but we only need the classical concept. 

By \cite{Pi:88} any definable group $G$ in $M$ has a natural group topology making it into a (regular) definable manifold, namely $G$ admits a finite atlas where each chart is definably isomorphic to an open subsets of $M^d$, with $d= \dim(G)$. When speaking about a definable group $G$ we will always assume that $G$ has the group topology rather than the topology induced by the ambient space $M^n$. So in particular a definable path in $G$ will always be assumed to be continuous with respect to the group topology of $G$. Note that since $M$ expands a field, by \cite[\S 10, Thm. 1.8]{vdD:98} we can always reduce to a situation in which the two topologies coincide, namely we can definably embed $G$ in $M^n$ as a subspace. In this case $G$ is {\bf definably compact} in the sense of \cite{PeSt:99} if and only if $G$ is closed and bounded in $M^n$. 

\section{Type-definability} 

We recall that a type-definable set is a set that can be presented as an infinite intersection $\bigwedge_{i\in I}X_i$ where each $X_i$ is definable and $I$ is a possibly infinite index set. Type-definable sets come equipped with a presentation, so it makes sense to interpret them in elementary extensions. Unlike what happens for definable sets, an equality $\bigwedge_{i\in I} X_i = \bigwedge_{j\in J} Y_j$ can hold in some model $M$ and fail in an elementary extension. To have a notion of equality not dependent on the model, we must restrict ourselves to models that are sufficiently saturated. For this reason it is convenient to identify a type-definable set with the set it defines in some big saturated ``monster model'' and insist that the index sets in the infinite conjunctions should be ``small'' (with respect to the saturation of the monster model). With these conventions one has for instance that infinite conjunctions commute with the existential quantifier, namely $\exists x \bigwedge_{i\in I} (x\in X_i) \equiv \bigwedge_{i\in I} \exists x (x\in X_i)$, provided the family of definable sets $\{X_i\}_{i\in I}$ is downward directed (e.g. it is closed under finite intersections). It is common practice to say that such equalities hold ``by saturation''. 
A $\bigvee$-definable set is a set presented as a union $\bigvee_{i\in I}X_i$, where each $X_i$ definable and the index set $I$ is small.  By saturation if a type-definable set is included in a $\bigvee$-definable set, there is some definable set between them. 

Starting with \cite{Pi:04}, type definability plays an important role in the study of definable groups in an o-minimal structure. Given a definable group $G$ and a type-definable subgroup $H<G$, one says that $H$ has {\bf bounded index} if the index $[G:H]$ is smaller than the amount of saturation of the monster model. Equivalently $[G:H]<2^{|T|+|A|}$ where $|T|$ is the cardinality of the language and $A$ is the set of parameters over which $G$ and $H$ are defined (\cite{Sh:08,HrPi:07}). If $H$ has bounded index, then $G/H$ does not depend on the model, in the sense that if $M$ is sufficiently saturated and $M' \succ M$, then the natural map $G(M)\to G(M')/H(M')$ has kernel $H(M)$ and therefore $G(M)/H (M) \cong G(M')/H (M')$ (\cite{Pi:04}). Each definably compact group $G$ in an o-minimal expansion of a field, has a smallest type-definable subgroup $G^{00} < G$ of bounded index, necessarily normal, and $G/G^{00}$ with the ``logic topology'', is a compact real Lie group (\cite{BeOtPePi:05}). Recall that a set $X \subset G/G^{00}$ is closed in the {\bf logic topology} if and only if its preimage in $G$ under the natural map $\tau^G \colon G \to \lie G$ is type-definable. It then follows that the preimage of an open subset of $\lie G$ is $\bigvee$-definable and that the image of a type-definable subset of $G$ is a closed subset of $\lie G$.  

\section{Topological consequences of compact domination}\label{top-cons}

One of the aims of this paper is to explore the topological consequences of the compact domination conjecture proved in \cite{HrPi:07} (and extended in \cite{HrPePi:08b} to the non-abelian case). One of its equivalent formulations says that given a definably compact group $G$, the image in $\lie G$ of a nowhere dense definable subset of $G$ has Haar measure zero. We will only use the following consequence of compact domination, proved in
\cite{Be:08}: the subgroup $G^{00}$ is a decreasing intersection $\bigcap_{i\in \NN} C_i$ of definably simply connected (actually even definably contractible) definable subsets $C_i$ of $G$. We will explore this situation in the following more general setting: 

\begin{assumption} \label{assumptions} 
Let $X$ be a subset of a definable set, let $Y$ be a locally simply connected second countable Tychonoff topological space, and let $f \colon X \to Y$ be a surjective function with the following properties:
\begin{enumerate}
\item The preimage of every closed subset of $Y$ is type-definable. 
\item For all $y\in Y$ the type-definable set $f^{-1}(y)$ is a decreasing intersection $\bigcap_{i\in \NN} C_i$ of definably simply connected definable subsets $C_i$ of $X$. 
\end{enumerate} 
\end{assumption}

Note that the second countability assumption ensures that $|Y|\leq 2^{\aleph_0}$, so in particular $Y$ is ``small'' (with respect to the saturation of the monster model). 

\bexa \label{mainex} The natural map $\tau \colon G \to \lie G$ satisfies Assumption \ref{assumptions} (where $G$ is a definably compact definable group). More generally let $O$ be an open subset of $\lie G$, and define $f\colon \tau^{-1}(O) \to O$ by  $f(x) = \tau(x)$.  Then $f$ satisfies Assumption \ref{assumptions}. 
\eexa
\bp $\lie G$ is a compact Lie group, so in particular it is Tychonoff, second countable and locally simply connected, and all these properties are inherited by its open subsets. For $y\in O$, $\tau^{-1}(y)$ is a coset of $G^{00}$, so it is a decreasing intersection $\bigcap_{i\in \NN} C_i$ of definably simply connected definable subsets $C_i$ of $G$. Since $O$ is open, $\tau^{-1}(O)$ is $\bigvee$-definable and therefore the sets $C_i$ are eventually contained in $\tau^{-1}(O)$. 
\ep

\bd (Definable fundamental groupoid) \label{fund}
Given a subset $X$ of a definable set, and a subset $\Gamma$ of $X$, let $P^{\rm{def}}(X,\Gamma)$ be the set of definable paths in $X$ with endpoints in $\Gamma$. Let $\dfund X \Gamma$ be the quotient of $P^{\rm{def}}(X,\Gamma)$ modulo definable homotopy of paths (relative endpoints). We define an operation $+$ on $\dfund X \Gamma$ by $[\alpha]+[\beta] = [\alpha + \beta]$ where $\alpha + \beta$ is the concatenation of the paths $\alpha$ and $\beta$ (with reparametrization).
This is  defined only when the final point of $\alpha$ coincides with the starting point of $\beta$. With this operation $\dfund X \Gamma$ is a groupoid, namely a category in which every morphism is an isomorphism. 
In particular when $\Gamma = X$ we obtain the {\bf definable fundamental groupoid} $\dfund X X$ of $X$. 
When $\Gamma$ is a singleton we obtain the definable fundamental group $\dfund X {x_0} := \dfund X {\{x_0\}}$, which will also be written as $\dfundg X$ when the base point is clear from the context or irrelevant. Dropping ``def'' one obtains the corresponding classical notions. 
\ed 

In this section we will prove the following theorem.

\bt \label{top-lemma} Let $f\colon X \to Y$ be as in Assumption \ref{assumptions}. Then there is a unique morphism of groupoids $f_* \colon \dfund X X \to \fund {Y} {Y}$ with the following properties: 
\begin{enumerate}
\item $f_* = f$ on the object part of the groupoids, namely $f_*$ maps the definable homotopy class of a path with endpoints $x,y$, to the homotopy class of a path with endpoints $f(x), f(y)$ respectively. 
\item\label{locality} For any open $O\subseteq Y$, and for any
$[a]\in\dfund{X}{X}$ such that $\img{a}\subseteq f^{-1}(O)$, there is a
path $b$ in $Y$ such that $\img{b}\subseteq O$ and $f_*([a])=[b]$.
\end{enumerate}
Moreover we have:
\begin{enumerate}
\item[(3)] If $\Gamma$ is a subset of $X$ such that $f \rest \Gamma$ is injective, then the restriction of $f_*$ to $\dfund X \Gamma$ is an isomorphism onto $\fund {Y} {f(\Gamma)}$.
\end{enumerate}
\et 

\br \label{cor}
In particular when $\Gamma$ is a singleton and $f$ is the projection $\tau \colon G \to \lie G$ we obtain an isomorphism $\tau_* \colon \dfundg G \stackrel{\cong}{\longrightarrow} \fundg {\lie G}$. 
Note that if $f\colon X \to Y$ satisfies Assumption \ref{assumptions} and $O$ is an open subset of $Y$, then the restriction $f^O \colon f^{-1}(O) \to O$ of $f$ continues to satisfy the assumption (we have restricted also the codomain to make the map surjective). So we have a morphism of groupoids $$f^O_* \colon \dfund {f^{-1}(O)}{f^{-1}(O)} \to \fund O O.$$ By (3), when $O$ is connected and $\Gamma$ is a singleton, we obtain an isomorphism $\dfundg {f^{-1}(O)} \to \fundg O$, hence the preimage $f^{-1}(O)$ of a simply connected open subset $O$ of $Y$ is definably simply connected. Moreover from the proof of Theorem \ref{top-lemma} it follows that we have a commutative diagram
\[
\pushout{\dfund {f^{-1}(O)}{f^{-1}(O)}} {\fund O O}{\dfund X X}{\fund Y Y}{f^O_*}{}{}{f_*}
\]
where the vertical arrows are induced by the inclusion maps.
\er

\begin{proof}[Proof of Theorem \ref{top-lemma}] The proof is splitted into a series of claims and definitions and will be completed at the end of this section. 

\begin{claim} \label{compact} Let $D$ be a definable subset of $X$. Then $f(D)$ is a compact subset of $Y$. The same conclusion holds if $D$ is only assumed to be type-definable. \end{claim} 

\begin{proof}[Proof of Claim.] 
The argument is well known: 
given a family $(Z_i \mid i\in I)$ of closed subsets of $f(D)$ with the finite intersection property, the preimages $f^{-1}(Z_i)$ form a small family of type-definable subsets of $D$ with the finite intersection property. By saturation they have a non-empty intersection in $D$, so the $Z_i$'s have a non-empty intersection in $f(D)$. \end{proof} 

\begin{claim}  \label{connected}
\begin{enumerate} 
\item Let $Z$ be a closed connected subset of $Y$. Then the type-definable set $f^{-1}(Z)$ is definably connected.  
\item Let $U$ be an open connected subset of $Y$. Then the
$\bigvee$-definable set $f^{-1}(U)$ is definably path-connected.
\end{enumerate} 
\end{claim} 
\begin{proof}[Proof of Claim.] (1) By \cite[Lemma 2.2]{BeOtPePi:05} if a type-definable set is the intersection of filtered family of definably connected sets, then it is itself definably connected. It then follows from Assumption \ref{assumptions} that for each $y\in Y$, the type-definable set $f^{-1}(y)$ is definably connected. Now let $Z$ be a closed connected subset of $Y$, and suppose for a contradiction that $f^{-1}(Z)$ is the union of two relatively definable disjoint non-empty open sets $A$ and $B$. Being relatively definable in a type-definable set, $A$ and $B$ are in fact type-definable. So their images $f(A)$ and $f(B)$ are closed (actually compact, by Claim \ref{compact}). Since $Z = f(A) \cup f(B)$ and $Z$ is connected, $f(A)$ and $f(B)$ have a non-empty intersection. Take $y\in f(A) \cap f(B)$. 
Then $f^{-1}(y)$ meets both $A$ and $B$, contradicting the fact that $f^{-1}(y)$ is definably connected.  \\
(2) Let $x,y\in f^{-1}(U)$. Choose a path $a$ in $U$ connecting $f(x)$ to $f(y)$. Then $Z:= \img (a)$ is a closed connected subset of $U$, so the type-definable set $f^{-1}(Z)$ is definably connected. Since this set is contained in the $\bigvee$-definable set $f^{-1}(U)$, by saturation there is a definable set $D$ with $f^{-1}(Z)\subset D \subset f^{-1}(U)$. The definably connected component $D'$ of $D$ containing $x$ must contain also $y$. Now it suffices to recall that a definable set is definably connected if and only if it is definably path connected. 
\end{proof}

\bd \label{star} 
An open cover $\ca P$ of a topological space is a {\bf star refinement} of a cover $\ca Q$, if for every $P\in \ca P$, there is a $Q \in \ca Q$ such that if $P'\in \ca P$ has a non-empty intersection with $P$ then $P' \subset Q$. In a metric space, and more generally in a  uniform space, every open cover has a star refinement. Every Tychnoff space admits a compatible uniform structure, so the existence of star refinements applies to Tychonoff spaces. In particular it applies to any subset of a compact Hausdorff space (since the Tychonoff condition is preserved in subspaces). 
\ed

\begin{claim} \label{covers}
There is an open cover $\ca U$ of $Y$ with the following properties. 
\begin{enumerate}
\item Each element of $\ca U$ is path connected, and whenever two elements of $\ca U$ have a non-empty intersection, their union is contained in some simply connected subset of $Y$. 
\item Let $\ca V := \{f^{-1}(U) \mid U\in \ca U\}$.  Each element of $\ca V$ is definably path connected (and $\bigvee$-definable), and whenever two elements of $\ca V$ have a non-empty intersection, their union is contained in some definably simply connected subset of $X$. 
\end{enumerate} 
\end{claim}

Note that it is easy to find a cover satisfying (1) (take a star-refinement of a cover by simply connected sets).  Moreover it will turn out that (2) is actually implied by (1). However for the moment we cannot assume this fact. 

\begin{proof}[Proof of Claim \ref{covers}] 
By Assumption \ref{assumptions} we can choose, for each $y\in Y$, a definably simply connected definable set $C_y \subset X$ containing $f^{-1}(y)$. Since $Y$ is locally simply connected and second countable, we can find a decreasing sequence $\{O_n \mid n\in \NN\}$ of simply connected open neighborhoods of $y$ with $\bigcap_n O_n = \{y\}$. Moreover since $Y$ is normal (being Tychonoff), we can arrange so that $\ov {O_{n+1}} \subset O_n$.  By saturation there is some $n$ such that $f^{-1}(\ov{O_n}) \subset C_y$ (using the fact that the preimage of a closed set is type-definable). 
Fix such an $n$ and let ${U'}_y := O_n$. Let $\ca U$ be a star-refinament of the cover ${\ca U}' := \{{U'}_y \mid y \in Y\}$. We can assume that each element of $\ca U$ is path-connected, as otherwise we could replace it by its connected components. So $\ca U$ satisfies (1). Now let ${\ca V}':=\{f^{-1}(U) \mid U' \in {\ca U}'\}$ and $\ca V := \{f^{-1}(U) \mid U\in \ca U\}$. Then $\ca V$ is a star-refinement of $\ca V'$. By the Claim \ref{connected} each member of $\ca V$ is definably path-connected. By construction whenever two members of $\ca V$ intersect, their union is contained in a set of the form $C_y$, which is a definably simply connected set. 
\end{proof} 

\begin{claim} Any $\bigvee$-definable open subset $V$ of $X$ is a small union of definable open sets. \end{claim}
\begin{proof}[Proof of Claim.]
Write $V$ as a small directed union $\bigcup_{i\in I}D_i$ of definable sets $D_i$. We claim that $V$ is the union  $\bigcup_{i\in I} \Int (D_i)$ of the interiors of the sets $D_i$. To this aim, let $x\in V$. Since $V$ is open there is a definable open neighbourhood $U$ of $x$ contained in $\bigcup_{i\in I}D_i$. Hence by saturation there is some $i\in I$ with $U\subset D_i$. But then $x \in \Int(D_i)$. 
\end{proof}

\bd \label{def-small} Let $\ca P$ be a family of subsets of a given set. A set is {\bf $\ca P$-small} if it is contained in some member of $\ca P$. A function is $\ca P$-small if its image is $\ca P$-small. \ed

\begin{claim} \label{subdivision}
\begin{enumerate}
\item Given a definable subset $D$ of $X$, there are finitely many sets in $\ca V$ whose union covers $D$. Moreover there is a finite partition of $D$ into definable sets whose closures are $\ca V$-small.  
\item Given a definable path $a$ in $X$ there is a subdivision $a = a_1 + \ldots + a_n$ of $a$ such that each $a_i$ is $\ca V$-small. 
\end{enumerate}
\end{claim}

\begin{proof}[Proof of Claim.] 

(1) $f(D)$ is a compact subset of $Y$, so it can be covered by finitely many members $U_1, \ldots, U_n$ of $\ca U$. It then follows that $D \subset V_1 \cup  \ldots \cup V_n$ with $V_i := f^{-1}(U_i) \in \ca V$. Each $V_i$ is $\bigvee$-definable and open, so by point (1) and saturation there are definable open subsets $O_i$ of $V_i$ with $D \subset O_1 \cup \ldots \cup O_n$. By shrinking each $O_i$ if necessary, we may assume that the closure of $O_i$ is included in $V_i$ for all $i$. It is now sufficient to take a cell decomposition of $D$ compatible with $O_1, \ldots, O_n$. 

(2) Take a finite partition $\ca P$ of $\img (a)$ into definable sets whose closures are $\ca V$-small. Then take a cell decomposition of $I = \dom(a)$ compatible with $f^{-1}(D)$ for all $D \in \ca P$. The endpoints of the decomposition yield the desired subdivision of $a$. 
\end{proof}

\bd 
\begin{enumerate} 
\item Given two definable paths $a$ and $b$ in $X$ with the same endpoints, we say that they are {\bf $\mathcal{V}$-contiguous} (written $a\sim_\mathcal{V}b$) if there are definable paths $u,v,a',b'$ in $X$ such that $a=u+a'+v$ and $b=u+b'+v$ and
$\img(a')\cup \img(b')\subset V$ for some $V \in\mathcal{V}$. Let
$\sim_\mathcal{V}^\star$ be the transitive closure of $\sim_\mathcal{V}$.
Two paths in the $\sim_\mathcal{V}^\star$ relation are said to be
{\bf $\mathcal{V}$-equivalent}.
\item Similarly, dropping all definability conditions, one defines the relation of $\ca U$-contiguity and $\ca U$-equivalence between paths in $Y$.
\end{enumerate}
\ed

\begin{claim} \label{contiguous} 
\begin{enumerate}
\item Two paths in $Y$ are $\ca U$-equivalent if and only if they are homotopic. 
\item Two definable paths in $X$ are $\ca V$-equivalent if and only if they are definably homotopic. 
\end{enumerate}
\end{claim}
\begin{proof}[Proof of Claim.]
One direction is trivial (equivalent implies homotopic). We prove the other direction. 

(1) Implicit in the proof of the van Kampen theorem in \cite{Br:68}.  
One argues as follows. Given a homotopy $F:I\times I \to Y$ from $a$ to $b$, we can subdivide the homotopy square $I\times I$ into small squares so that each is mapped by $F$ into an open set of $\ca V$. If $n$ is the number of squares in the subdivision of $I\times I$, it is easy to see that $a$ is $\ca V$-equivalent to $b$ by a sequence of $n$ contiguity moves. The converse is trivial, since two $\ca U$-contiguous paths differ for a subpath contained in a simply connected set. 

(2) Let  $F\colon I\times I \to X$ be a definable homotopy between the definable paths $a=F(0,\cdot)$ and $b= F(1,\cdot)$ in $X$. By Claim \ref{subdivision} the image of $F$ can be partitioned into finitely many definable sets $D_1, \ldots, D_n$ whose closures are $\ca V$-small. Consider a cell decomposition of the homotopy square $I\times I$ compatible with $F^{-1}(D_i)$ for all $i$. We can then reason as above with the role of the small squares be replaced by the cells of the decomposition. For the details see the proof of the o-minimal van Kampen theorem in \cite{BeOt:02}. 
\end{proof}

We are now ready to define $f_* \colon \dfund X X \to \fund {Y} {Y}$.

\bd \label{defi} On the object part of the groupoids we set $f_*=f$. 
Given a definable path $a$ in $X$ and a path $b$ in $Y$ we say that $a$ corresponds to $b$ if there is a subdivision $a=a_1+\ldots + a_n$ into $\ca V$-small definable paths $a_i$, and a subdivision $b=b_1+\ldots + b_n$ into $\ca U$-small paths such that, for all $i\leq n$, the endpoints of $a_i$ are mapped by $f\colon X \to Y$ to the respective endpoints of $b_i$. If $a$ corresponds to $b$, then we define $f_*([a]) = [b]$. 
\ed 

\begin{claim}
$f_*$ is well defined. 
\end{claim}
\begin{proof}[Proof of Claim.]
Step 1. First note that two $\ca V$-small definable paths $a,a'$ in $X$ with the same endpoints are definably homotopic. In fact if $V\in \ca V$ and $V'\in \ca V$ contain the images of $a,a'$ respectively, then by Claim \ref{covers} $V \cup V'$ is contained in a definably simply connected set, and therefore $a$ is definably homotopic to $a'$. Similarly two $\ca U$-small paths $b,b'$ in $Y$ with the same endpoints are homotopic. So $f_*([a]) = [b]$ is well defined at least when $a,b$ are small. 

Step 2. We next show that if $a$ corresponds to $b$, then the homotopy class of $b$ is determined by $a$. So suppose that $a = a_1+\ldots + a_n$ is a subdivision of $a$ into $\ca V$-small definable paths, let $a_i$ correspond to $b_i$ (a path in $Y$), and let $b= b_1+\ldots +b_n$. By Step 1 the homotopy class of each $b_j$ is determined by the corresponding $a_j$, but we must prove that the homotopy class of $b$ does not depend on the chosen subdivision of $a$. Since any two subdivisions have a common refinement, it suffices to consider the case in which one of the $a_i$ is further subdivided into $\ca V$-small paths. So without loss of generality suppose $i=1$ and let $a_1 = a_1^1+\ldots + a_1^k$ be a subdivision of $a_1$ into $\ca V$-small paths. We must show that $b_1$ is homotopic to $b_1^1+\ldots+b_1^k$ where each $b_1^j$ is such that $a_1^j$ corresponds to $b_1^j$. To this aim let $U\in \ca U$ be such that $\img (a_1) \subset f^{-1}(U)$. Then in particular the endpoints of $a_1$ and of each $a_1^j$ are in $f^{-1}(U)$. Therefore the endpoints of $b_1$ and of each $b_1^j$ are in $U$. Reasoning as in Step 1  we can then assume that the image of $b_1$ and of each $b_1^j$ is entirely contained in $U$ since we can reduce to this case replacing each of these paths by a homotopic path. After these reductions, $b_1$ and $b_1^1+\ldots + b_1^k$ are two paths in $U$ with the same endpoints, and therefore they are homotopic (since $U$ is contained in a simply connected set). 

Step 3. Finally suppose that $a$ is definably homotopic to $a'$, and let us show that $b$ is homotopic to $b'$, where $a$ corresponds to $b$ and $a'$ to $b'$. 
By Claim \ref{contiguous} we can assume that $a'$ is $\ca V$-contiguous to $a$. So we can write $a= u+z+v$ and $a' = u + z' + v$ with $\img (z) \cup \img (z') \subset V$ for some $V\in \ca V$. Choose subdivisions of $a$ and $a'$ such that $z$ and $z'$ are segments of the chosen subdivisions. By Step 2 we may assume that $b$ and $b'$ are obtained from $a,a'$ using these subdivisions. So we can assume that $b$ and $b'$ are $\ca U$-contiguous. Therefore, by Claim \ref{contiguous}, they are homotopic.   
\end{proof} 

\begin{claim} \label{u-indep} The definition of $f_*$ does not depend on the particular choice of the cover $\ca U$ in Claim \ref{covers}. 
\end{claim}
\begin{proof}[Proof of Claim] 
Suppose $\ca U'$ is a refinement of $\ca U$ still satisfying the conditions in Claim \ref{covers}. Then clearly if we define $f_*([a])$ using $\ca U'$ instead of $\ca U$ we get the same function (since a subdivision of $a$ compatible with the preimages of the sets in $\ca U'$ is also compactible with the preimages of the sets in $\ca U$). Now it it suffices to observe that for any two coverings $\ca U$ and $\ca U'$ satisfying the conditions in Claim \ref{covers}  there is a common refinement which still satisfies the conditions (take the connected components of the pairwise intersections of an element of $\ca U$ and an element of $\ca U'$). 
\end{proof}  

\begin{claim}
$f_*$ is a morphism and satisfies (1) and (\ref{locality}) in Theorem \ref{top-lemma}.
\end{claim}

\begin{proof}[Proof of Claim] Note that in Definition \ref{defi}, if $a=a_0 + a_1$ and $a_0,a_1$ correspond to $b_0,b_1$ respectively, then $a$ corresponds to $b:=b_0+b_1$. Hence $f_*$ is a morphism.  
By construction it satisfies (1). We prove (2). 
Because of Claim \ref{u-indep}, by enlarging $\mathcal{U}$ we can suppose it to be a base of the topology
of $Y$. Given an open subset $O$ of $Y$, we can then express $O$ as the union of a subfamily $\mathcal{U}'$ of
$\mathcal{U}$. Consider
$\mathcal{V}'=\left\{f^{-1}(U):U\in\mathcal{U}'\right\}$. Given a path $a$
in $f^{-1}(O)$, the construction of $[b]=f_*([a])$ can be carried out subdividing
$a$ into $\mathcal{V}'$-small paths and associating to each of them a
$\mathcal{U}'$-small path in $Y$. The image of $b$ is clearly a subset of
$O$.
\end{proof}

\begin{claim}  \label{uniqueness} $f^*$ is the unique morphism satisfying (1) and (2) in Theorem \ref{top-lemma}. \end{claim} 
\begin{proof}[Proof of Claim] Let $O$ be a simply connected open subset of $Y$. 
If $a$ is a definable path in $X$ with $\img{a}\subseteq f^{-1}(O)$, condition (\ref{locality}) in Theorem \ref{top-lemma} forces $f_*([a])$ to be of the form $[b]$ for some path $b$ with $\img{b}\subseteq O$. Since $O$ is simply connected, and the endpoints of $b$ are the images of the endpoints of $a$, $[b]$ is completely determined. 
So $f_*$ is determined on the paths satisfying $\img (a) \subseteq f^{-1}(O)$ for some open simply connected subset $O\subset Y$. By Claim \ref{subdivision} we can reduce to this situation by subdividing the paths. 
\end{proof}

\begin{claim} \label{inverse}
Let $\Gamma\subset X$ be such that $f \rest {\Gamma}\colon \Gamma \to Y$ is injective. Then $f_* \rest {\fund X \Gamma}:\dfund X \Gamma\to\fund {Y} {f(\Gamma)} $ is an isomorphism.
\end{claim}

\begin{proof} Choose a right inverse $\psi \colon Y \to X$ of $f\colon X\to Y$ extending $(f\rest \Gamma)^{-1}$.  
We define an inverse $\psi_*:\fund {Y} {f(\Gamma)} \to\dfund X \Gamma$ of $f_*$ as follows. Given $[b]\in\fund {Y} {f(\Gamma)}$, consider a subdivision $b=b_1+\cdots+b_m$ of
$b$ into $\mathcal{U}$-small paths. For each $i$, let $y_{i-1}$ and $y_i$
be the endpoints of $b_i$, 
and let $a_i$ be a $\mathcal{V}$-small path in
$X$ from $x_{i-1}=\psi(y_{i-1})$ to $x_i=\psi(y_i)$.
Finally define $\psi_*([b])=[a_1+\cdots+a_m]$.
Note that $\psi_*$ is well defined by the same argument that
proves that $f_*$ is well defined. We also claim that
$\psi_*=(f_* \rest {\dfund X \Gamma})^{-1}$. In fact, by inspection of the
definitions, the same pair of subdivisions $a=a_1+\cdots+a_m$ and
$b=b_1+\cdots+b_m$ witnesses both $f_*([a])=[b]$ and 
$\psi_*([b])=[a]$ simultaneously.  
\end{proof} 

The proof of Theorem \ref{top-lemma} is now complete. 
\end{proof}

\section{Equivariance} 

We now specialize the results of the previous section to the case when the spaces carry a group structure. So let $G$ be a definably compact group. Let $\tau^G\colon G \to \lie G$ be the natural map, and let $\tau^G_* \colon \dfund G G \to \fund {\lie G} {\lie G}$ be the induced morphism of groupoids as in Theorem \ref{top-lemma}. 

\bd The group 
$G$ acts on $\dfund G G$ by $x\cdot [a] = [x\cdot a]$, where $x\in G$ and $a$ is a definable path in $G$. Similarly we have an action of $\lie G$ on $\fund {\lie G}{\lie G}$ given by $y \cdot [b]:= [y\cdot b]$. Finally we have an action of $G$ on $\fund {\lie G} {\lie G}$ sending $(x, [b])$ to $\tau^G(x)\cdot [b] $, where $x \in G$ and $b$ is a path in $\lie G$. 
\ed

\bt \label{equivariant} $\tau^G_* \colon \dfund G G \to \fund {\lie G} {\lie G}$ is equivariant under the action of $G$, namely for each $x\in G$
and $[a]\in \dfund G G$, we have $\tau^G_* (x \cdot [a]) = \tau^G(x) \cdot \tau^G_*([a])$. \et

\begin{proof}
Let $x\in G$ and consider the map 
$\tau_x\colon\dfund G G \to \fund{\lie G}{\lie G}$ sending $[a]$ to $\tau^G(x^{-1}) \cdot \tau^G_*(x\cdot [a])$. 
It is easy to check that $\tau_x$ is a groupoid morphism satisfying the conditions (1) and (\ref{locality}) of Theorem~\ref{top-lemma}. Hence by the Theorem $\tau_x=\tau^G_*$ for all $x\in G$. It follows that $\tau^G_*$ is equivariant. \end{proof}

\section{Functors and natural transformations}

In this section we establish the funtoriality properties of the morphisms $\tau^G_*$. We can regard the correspondence $G\mapsto \dfund G G$ as the
object part of a functor $\pi^{\rm{def}}$ from the category of definably compact groups
(and definable group homomorphisms) to the category of groupoids (and groupoid homomorphisms).
Similarly we have a functor $\pi_1^{\rm{def}} \colon G \mapsto \dfundg G$ from definably compact groups to groups. Finally we have a functor $F\colon G \to \lie G$ from definably compact groups
to compact Lie groups (and Lie homomorphisms). By Theorem \ref{top-lemma}, 
for $G$ a definably compact group, the projection $\tau^G \colon G \to \lie G$ induces a morphism $$\tau^G_* \colon \dfund G G \to \fund {\lie G} {\lie G}.$$ 

\bt \label{comparison} The family $\tau_* = (\tau^G_*)_G$ is a natural transformation of the functor
$\pi^{\rm{def}} \colon G \mapsto \dfund G G$ to the functor
$\pi \circ F \colon G \mapsto \fund {\lie G}{\lie G}$.
In other words, given a definable morphism $f \colon G\to G'$ we have a commutative diagram in the category of groupoids: 
\[
\pushout{\dfund G G}{\dfund {G'} {G'}}{\fund {\lie G} {\lie G}}{\fund {\lie {G'}} {\lie {G'}}}{\pi^{\rm{def}}(f)}{\tau^G_*}{\tau^{G'}_*}{\pi(F(f))}
\]
where $F(f)\colon \lie G \to \lie {G'}$ is the induced Lie homomorphism. 
\et 

\bp
Consider an open cover ${\ca U}'$ of $\lie {G'}$ by simply connected sets. Now consider an open cover $\ca U$ of $\lie {G}$ by simply connected sets which refines $\{{F(f)}^{-1}(U')\mid U'\in {\ca U}' \}$. Using these covers in Definition \ref{defi}, the commutativity of the diagram follows immediately. 
\ep 

\br \label{isom} By \cite{HrPePi:08} $G^{00}$ is torsion free (see \cite{Be:07} for the non-abelian case), so if $\Gamma$ is a finite subgroup of $G$, then $\tau^G$ maps $\Gamma$ isomorphically onto its image in $\lie G$. It then follows by point (3) in Theorem \ref{top-lemma}, that the restriction of $\tau^G_*$ to $\dfund{G}{\Gamma}$ is an isomorphism onto $\fund{\lie G}{\tau^G(\Gamma)}$. 
\er

\section{Universal cover} 
Let $G$ be a definably compact definably connected definable group. 
The (o-minimal) universal cover $\uc G$ of a definable group $G$ has been studied in \cite{EdEl:07}. We 
can identify $\uc{G}$ with a subset of the groupoid $\dfund G G$, namely we let 
\begin{equation*} \uc G \subset \dfund G G \end{equation*}
be the subset consisting of all the definable homotopy classes of paths starting at $1\in G$. The group operation of $\uc{G}$ is defined by $[a]*[b] := [a]+[a(1)\cdot b]$ where $a(1)$ is the endpoint of the definable path $a$ and $``+"$ is the operation of the groupoid (induced by concatenation of paths).  There is a morphism of groups from $\uc{G}$ to $G$ sending $[a]$ to $a(1)$. We endow $\uc{G}$ with a group topology as follows: if $U$ is a definably simply connected open subset of $G$, then the set of all $[a]\in \uc{G}$ such that $\img(a)\subset U$ is a basic open neighbourhood of $1$ in $\uc{G}$. By left translation we obtain the basic open neighbourhoods around the other points. With this topology $\uc{G}$ is a locally definable space in the sense of \cite{BaOt:09b} (it can also be seen as an inductive limit of definable sets with the final topology), and the morphism $[a]\in \uc{G} \mapsto a(1)\in G$ is a local homeomorphism. With analogous definitions the universal cover of the real Lie group $\lie G$ can be identified with the subset 
\begin{equation*} 
\uc {\lie G} \subset \fund {\lie G}{\lie G}
\end{equation*}
consisting of all the homotopy classes of paths starting at $1\in \lie G$. 

\bd The morphism of groupoids $\tau^G_* \colon \dfund G G \to \fund {\lie G}{\lie G}$ given by Theorem \ref{top-lemma} induces a map $$\rho^G \colon \uc{G} \to \uc{\lie G}$$ 
by restriction, namely $\rho^G([a]) := \tau^G_*([a])$. 
\ed 

Since ${\tau^G_*}$ is equivariant (Theorem \ref{equivariant}), it follows at once that $\rho^G$ 
is a morphism of groups. Moreover we have: 

\bt \label{uc} Given a definably compact definably connected definable group $G$, the kernel of ${\rho^G} \colon \uc{G}\to \uc{\lie G}$ is isomorphic to $G^{00}$ via the map $[a] \mapsto a(1)$. \et 

\bp Let $[a]$ be in the kernel. Then $\tau^G_*([a]) = [b]$ where $b$ is a contractible loop at $1\in {\lie G}$. Since $\tau^G$ must send the endpoints of $a$ to the endpoints of $b$ (namely to the identity of ${\lie G}$), it follows that $a(1)\in G^{00}$. So we have a well defined map from $\ker (\rho^G)$ to $G^{00}$ which moreover is a group homomorphism. 
We must prove that it is an isomorphism. 

(Surjectivity) Given $x\in G^{00}$ we must find $[a]\in \ker (\rho^G)$ with $a(1)=x$. To this aim let $U$ be a a simply connected open neighbourhood of $1$ in $\lie G$. By Remark \ref{cor},  $V:= {\tau^G}^{-1}(U)$ is a ($\bigvee$-definable) definably simply connected subset of $G$. In particular $V$ is definably path connected, so there is a definable path $a$ in $U$ from $1 \in G$ to $x$. By Theorem \ref{top-lemma} there is a path $b$ in $U$ with $\tau^G_* ([a]) = [b]$. Since $x \in G^{00}$, the endpoint $\tau^G(x)$ of $b$ is the identity of $\lie G$, namely $b$ is a loop. Moreover since the image of $b$ is contained in the simply connected set $U$, $[b]$ is the identity of $\uc{\lie G}$, and therefore $[a]\in \ker (\rho^G)$. 

(Injectivity) Let $[a] \in \ker (\rho^G)$ and suppose $a(1) = 1_G$ (namely $a$ is a loop).  
By Theorem \ref{isom} $\tau^G_*$ sends $\dfundg G$ bijectively to $\fundg {\lie G}$. Since under this map $[a]$ goes to the identity, it follows that $a$ is definably contractible, namely $[a]$ is the identity of $\uc{G}$.  
\ep 

By the above theorem we have a commutative diagram 
\begin{equation*}
\pushout{\uc G}{\uc {\lie G}}{G}{\lie G}{}{}{}{}.
\end{equation*}
As a consequence we obtain similar diagrams for the finite extensions of $\lie G$:

\begin{proposition} Let $G$ be a definably compact definably connected definable group. Given an extension of Lie groups $f\colon \mathbf H \to \lie G$ with a finite kernel, there is a  definable group extension $\pi \colon H \to G$ of $G$ such that $\lie H \cong \mathbf H$ (as coverings of $\lie G$). We thus obtain a commutative diagram
\begin{equation*}
\pushout{H}{\mathbf H }{G}{\lie G}{\varphi}{\pi}{f}{\tau}
\end{equation*}
where $\varphi \colon H \to \mathbf H$ is the composition of the projection $H\to \lie H$ with the isomorphism $\lie H \cong \mathbf H$. 
\end{proposition}
\bp Note that with our definitions $\dfundg {\lie G}$ is a discrete central subgroup of $\uc{\lie G}$.
Let $L \lhd \uc{\lie G}$ be the image of $\pi_1(f)\colon \fundg {\mathbf H} \to \fundg {\lie G}$. 
We can identify $\mathbf H$ with $\uc{\lie G}/L$ . Let $L'$ be the preimage of $L$ under the isomorphism $\dfundg G \to \fundg {\lie G}$ of Theorem \ref{top-lemma}.  Define $H = \uc{G}/L'$. The projection $\uc{G} \to G$ induces, passing to the quotient, a morphism $\pi \colon H\to G$ whose kernel $\Gamma= \dfundg G /L'$ is isomorphic to the kernel $\fundg {\lie G}/L$ of $f\colon \mathbf H \to {\lie G}$. 
Given a definable simply connected subset $U$ of $G$ there is a homeomorphism
$U \times \Gamma \to \pi^{-1}(U)$ where $H = \uc{G}/L'$ has the quotient topology.  
Since $G$ can be covered by finitely many definably simply connected subsets, this gives us an interpretation of $\pi \colon H \to G$ in the underlying o-minimal structure. In other words, up to isomorphism, $H$ and $\pi$ are definable. 
Now let $\uc{G}^{00}$ be the kernel of the morphism $\rho^G \colon \uc{G} \to \uc{\lie G}$ of Theorem \ref{uc}. Since $\rho^G(L')=L$, $\rho^G$ induces a morphism $\varphi \colon H \to \mathbf H$ whose kernel is easily verified to be the image of $\uc{G}^{00}$ under the quotient map $\uc{G}\to H$. Under the above interpretation of $H$, this kernel coincides with $H^{00}$. 
\ep 

\section{Definably compact semisimple groups}

Let $G$ be a definably compact definably connected semisimple group. 
In this section we show that the (Lie-)isomophism type of $\lie G$ detemines the definable isomorphism type of $G$.  

\begin{lemma}\label{deflift}
Work in an o-minimal expansion of $\R$. 
Let $f:X\to B$ be a definable continuous map. 
Let $p:E\to B$ be a definable covering map. 
Let $\w{f}:X\to E$ be a lifting of $f$ (i.e. a
continuous function, not necessarily definable, such that
$p\circ\w{f}=f$). Then $\w{f}$ is definable. 
\end{lemma}
\begin{proof}
By definition of definable covering, we have a definable finite cover $\mathcal{U}$ of $B$ by definably connected definable open sets
such that, for any $U\in\mathcal{U}$, the inverse image $p^{-1}(U)$ is a
finite disjoint union of definably connected open sets on each of which $p$ is an
homeomorphism onto $U$.  Fix $U\in\mathcal{U}$ and let $E_1,\cdots,E_m$
be the definably connected components of $p^{-1}(U)$ and $X_1,\cdots,X_n$ be the
definably connected components of $f^{-1}(U)$. Note that all these sets are definable. 
Fix an $i\in\{1,\cdots,n\}$. Since we are working over $\R$, a definably connected set is connected. So, by
continuity of $f$, there is a $j\in\{1,\cdots,m\}$ such that
$\w{f}(X_i)\subset E_j$. Hence $(\w{f} \rest {X_i})(x)=y$ if and only if
$x\in X_i \wedge y\in E_j \wedge f(x)=p(y)$. This proves that
$\w{f} \rest {X_i}$ is definable, and the definability of $\w{f}$
follows observing that the same hold for any $U\in\mathcal{U}$ and any
$i$.
\end{proof}

\begin{fact}\label{semisimple-semialgebraic}(\cite[theorem 3.1]{EdJoPe:07} or 
\cite[theorem 4.4 (ii)]{HrPePi:08b}).
For any semisimple definable group $G$, there is a group $G'$,
semialgebraic without parameters, definably isomorphic to it.
\end{fact}

\begin{lemma}\label{get-semialgebraic-isomorphism}
Let $G$ and $G'$ be definably connected semialgebraic semisimple groups defined over $\R$. By \cite{Pi:88} $G(\R)$ and $G'(\R)$ have a natural Lie group structure. 
Suppose that $f:G(\R) \to G'(\R)$ is a Lie isomorphism. Then $f$ is semialgebraic over $\R$. 
\end{lemma}
\begin{proof}
We first prove the result under the additional assumption that $G$ and $G'$ are centerless. 
The isomorphism $f\colon G(\R)\to G'(\R)$ induces an isomorphism $\phi:\mathfrak{g}\to\mathfrak{g}'$ of the corresponding Lie algebras. Since we are in the centerless case, the adjoint representation $\Ad_G:G(\R) \to\Aut(\mathfrak{g})$ is an isomorphism onto $\Aut^0(\mathfrak{g})$ and similarly for $G'(\R)$. Fixing a basis of the vector spaces $\mathfrak g$ and $\mathfrak g'$, we can consider $\Ad_G$ and $\Ad_G'$ as semialgebraic maps. Let $\w \phi \colon \Aut^0(\mathfrak g) \to \Aut^0 ({\mathfrak g}')$ be the isomorphism induced by $\phi$. Then $f=\Ad\circ\w\phi\circ\Ad'^{-1}$ and therefore $f$ is semialgebraic over $\R$.  

To reduce the general case to the centerless case we use the fact that $G/Z(G)$ and $G'/Z(G')$ are centerless. Clearly $f$ induces an isomorphism $g \colon G/Z(G) \to G'/Z(G')$. By the centerless case $g$ is semialgebraic. By Lemma \ref{deflift}, $f$ is itself semialgebraic.  
\end{proof}

\br In the above Lemma we cannot ensure that $f$ is semialgebraic over $\R^{\rm{alg}}$ even assuming that $G$ and $G'$ are semialgebraic over $\R^{\rm{alg}}$. In fact let $G = G' = SO(3,\R)$. The group of inner automorphisms of $SO(3,\R)$ is non-trivial and connected, so it has the cardinality of the continuum. Therefore there is some inner automorphism $f\colon SO(3,\R) \to SO(3,\R)$ which is not definable over $\R^{\rm{alg}}$. \er

\bt\label{semisimple-case}
Let $G$ and $G'$ be definably compact definably connected 
semisimple definable groups.
Suppose that there is a Lie isomorphism
$\psi \colon\lie G \to \lie {G'}$.
Then there is a definable isomorphism $f \colon G\to G'$. 
If the o-minimal structure is sufficiently saturated we can consider the projection $\tau^G\colon G\to\lie G$ and $\tau^{G'}\colon G'\to\lie{G'}$ and we can choose $f$ so that 
$\tau^{G'}\circ f  = \psi \circ\tau^G$. 
\et
\begin{proof} 
By fact \ref{semisimple-semialgebraic} we may assume $G$ and $G'$ to be
semialgebraic without parameters. So it makes sense to consider the groups $G(\R)$ and $G'(\R)$. If $M$ is sufficiently saturated there is an elementary embedding of $\R$ into $M$ (in the language of fields) and there is a surjective homomorphism $G(M)\to G(\R)$ (given by the ``standard part map'') whose kernel is $G^{00} = G^{00}(M)$ (\cite{Pi:04}). Similarly for $G'$. So $\lie{G} \cong G(\mathbb{R})$ and $\lie{G'} \cong G'(\mathbb{R})$ (with the logic topology). Hence
we have a Lie isomorphism $\psi' \colon G(\mathbb{R}) \to G'(\mathbb{R})$ induced by $\psi$. 
By lemma \ref{get-semialgebraic-isomorphism} $\psi'$ is semialgebraic over $\R$. 
The same formula defines an isomorphism $f \colon G(M)\to G'(M)$ with $\tau^{G'}\circ f  = \psi \circ\tau^G$. If $M$ is not sufficiently saturated, then we can go to a saturated extension $M'$ to get an $M'$-definable isomorphism $f \colon G(M')\to G'(M')$ as above, and therefore (quantifying over the parameters) also an $M$-definable isomorphism from $G(M)$ to $G'(M)$.  
\end{proof}

\section{Definably compact abelian groups} 

In this section we try to understand, in the abelian case, up to which extent $\lie G$ determines $G$.
It is known that there are definably compact definably connected abelian groups $G$ and $G'$ of the same dimension (hence with $\lie G \cong \lie {G'}$) which are not definably isomorphic (see \cite{St:94,PeSt:99}). Indeed it may happen that $G$ splits as a product of $1$-dimensional definable subgroups, while $G'$ has no definable $1$-dimensional subgroups  
However by \cite{BeMaOt:09} any two definably compact definably connected abelian groups of the same dimension are definably homotopy equivalent. The same proof yields the following: 

\bl \label{abelian-homotopy} (See \cite[Theorem 3.4]{BeMaOt:09}) Let $G,G'$ be definably compact definably connected abelian groups of the same dimension $n$. (So $\fundg G \cong \fundg {G'} \cong \Z^n$ by \cite{EdOt:04}.) Let $\theta \colon \fundg G \to \fundg {G'}$ be an isomorphism. Then there is a definable continuous map $f\colon G\to G'$ with $\pi_1(f) = \theta$ and $f(e) = e$. Moreover, any such map $f$ is a definable homotopy equivalence.  
\el 

\bp 
Special case. Suppose that $G$ is the direct product of $1$-dimensional definable subgroups. 
Choose free generators $[a_1], \ldots, [a_n]$ of $\dfundg G$ such that each $x\in G$ can be written uniquely in the form $x= a_1(t_1)+\ldots + a_n(t_n)$ with $0\leq t_i < 1$. 
Let $[b_1], \ldots, [b_n]\in \dfundg {G'}$ be the images of $[a_1], \ldots, [a_n]$ under $\theta$. Define $f(x) = b_1(t_1)+\ldots + b_n(t_n)$. Then clearly $\pi_1^{\rm{def}}(f) = \theta$. Since the higher definable homotopy groups of $G$ and $G'$ are zero by \cite{BeMaOt:09}, $f$ is a definable homotopy equivalence by the o-minimal version of Whitehead theorem in \cite{BaOt:09}. 

General case. We reduce to the special case as follows. 
Let $T$ be a definably compact definably connected one-dimensional abelian group, and let $T^n$ be the direct product of $n$-copies of $T$, where $n= \dim G$. By \cite{EdOt:04}, $\dfundg {G} \cong \fundg {T^n} \cong \Z^n$. Choose an isomorphism $\lambda\colon \dfundg {T^n} \to \dfundg G$. Then $\theta \circ \lambda \colon \dfundg {T^n} \to \dfundg {G'}$ is an isomorphism. By the special case we get definable homotopy equivalences $g,h$ with $\pi_1(g) = \lambda$ and $\pi_1(h) = \theta \circ \lambda$. Let $g'$ be a definable homotopy inverse of $g$. So $f := h \circ g'$ satisfies $\pi_1(f) = \theta$. 
\ep 

To improve on the above result we need a definition. 

\bd \label{equiva-riant-lence} Let $G$ and $G'$ be definable groups with  subgroups $\Gamma < G$ and $\Gamma' < G'$.
We say that a definable continuous map $f\colon G \to G'$ is a {\bf $\Gamma$-equivariant definable homotopy equivalence} if $f\rest \Gamma$ is an isomorphism onto $\Gamma'$ and $f$ admits a definable homotopy inverse $f'$ such that the following holds: 
\begin{itemize}
\item $f(1)=1$ and $f(cx)=f(c)f(x)$ for any $c\in\Gamma$ and $x\in G$;
\item $f'(1)=1$ and $f'(c'x')=f'(c')f'(x')$ for any $c'\in\Gamma'$ and any $x'\in G'$;
\item there is a definable homotopy $h:I\times G\to G$ relative to $\Gamma$
between $f'\circ f$ and the
identity on $G$ such that $h_t(cx)=ch_t(x)$ for any $c\in\Gamma$, 
$x\in G$, and $t\in I$;
\item there is a definable homotopy $h':I\times G'\to G'$ relative to $\Gamma'$
between $f\circ f'$ and the
identity on $G'$ such that $h'_t(c'x')=c'h'_t(x')$ for any
$c'\in\Gamma'$, $x'\in G'$, and $t\in I$.
\end{itemize}
Note that $f'\rest {\Gamma'}$ is the inverse of $f\rest \Gamma$. 
\ed 

To prove the existence of $\Gamma$-equivariant homotopy equivalences we need some preliminary results. The following Lemma says that given a definable covering map $p\colon E \to B$ we can always lift a definable continuous function $f\colon X \to B$ to a function $\w f \colon X \to E$ provided there are no obstructions coming from the fundamental group. 

\bl \label{lifting-exists}
Let $p:E\to B$ be a definable covering map, with $B$ definably connected.
And let $f:X\to B$ be a definable continuous map from a definable
definably connected set $X$ to $B$.  Fix base points $e_0 \in E$, $b_0\in B$ and $x_0 \in X$
with $f(x_0)=p(e_0)=b_0$. Consider the homomorphisms $\pi_1(p)$ and $\pi_1(f)$ induced by $p$ and $f$ on the
definable fundamental groups. If
$\img\pi_1^{\rm{def}}(f)\subset\img\pi_1^{\rm{def}}(p)$ then there
is a unique definable continuous function $\w{f}:X\to E$ lifting  $f$ (i.e. such
that $p\circ\w{f}=f$) with $\w{f}(x_0)=e_0$.
\el 
\begin{proof}
The proof of the corresponding classical result (see \cite[Theorem 2.4.5]{Sp:66}) can be adapted to the o-minimal category thanks to the definable version of the homotopy lifting property in \cite{BaOt:09}. More precisely, for each $x \in X$ choose, uniformly in $x$, a definable path $a_x$ from $x_0$ to $x$ in $X$. Then $b_x:= f \circ a_x$ is a definable path in $B$. Let $\w {b_x}$ be its (unique) lifting to a definable path in $E$ with starting point $e_0$. Define $\w f (x)$ as the final point of $\w {b_x}$. This is independent on the choice of the paths and works. 
\end{proof}

\bt \label{abelian-case}
Let $G, G'$ be definably compact definably connected abelian groups. Let 
\[\psi \colon \lie G \to \lie {G'}\] be an isomorphism of Lie groups.  Let
$\Gamma$ be a finite subgroup of $G$. Let $\Gamma'$ be the unique finite sugroup of $G'$ such that $\Gamma' {G'}^{00}=\psi(\Gamma G^{00})$. Then there is a $\Gamma$-equivariant definable homotopy equivalence $f^G \colon G \to G'$ which coincides with $\psi$ on $\Gamma$ (more precisely for each $c\in \Gamma$ $f^G(c){G'}^{00} = \psi(cG^{00})$). 
\et 

\bp
To simplify notations let $\mathbf G := \lie G$ and $\mathbf {G'}:= \lie {G'}$. 
Since $\Gamma$ is finite the projection $G \to \mathbf G$ maps $\Gamma$ isomorphically onto its image ${\mathbf \Gamma} := \Gamma G^{00}/G^{00} < \mathbf G$. Let $\mathbf {\Gamma'} := \psi(\mathbf {\Gamma}) < \mathbf {G'}$ and let $\Gamma'$ be the unique finite sugroup of $G'$ which is mapped to $\mathbf \Gamma'$ under the projection $G' \to \mathbf {G'}$. 

Passing to the quotient, the isomorphism $\psi \colon \mathbf G \to \mathbf {G'}$ induces an isomorphism $\phi \colon \mathbf {G/\Gamma} \to \mathbf {G'/\Gamma'}$ making the following diagram commute (where the vertical arrows are the projections):
\begin{equation*}
\pushout{\mathbf G}{\mathbf {G'}}{\mathbf{G/\Gamma}}{\mathbf{G'/{\Gamma'}}}{\psi}{}{}{\phi}
\end{equation*}
Since $\mathbf \Gamma$ and $\mathbf {\Gamma'}$ are mapped to the identity of $\mathbf {G/\Gamma}$ and $\mathbf {G'/\Gamma'}$ respectively, we  obtain an induced commutative diagram in the category of groupoids: 
\begin{equation*}\tag{*} \label{4}
\pushout{\fund{\mathbf G}{\mathbf \Gamma}}{\fund{\mathbf G'}{\mathbf {\Gamma'}}}{\fund{\mathbf{G/\Gamma}}1}{\fund{\mathbf{G'/\Gamma'}} 1}{\pi(\psi)}{}{}{\pi(\phi)} 
\end{equation*}
where $\mathbf {G/\Gamma}$ can be naturally identified with $(G/\Gamma)/(G/\Gamma)^{00}$ and similarly on the $G'$ side. By Theorem \ref{comparison} and Remark \ref{isom} we have commutative diagrams 
\begin{equation*} 
\pushout {\dfund G \Gamma} {\dfund {G/\Gamma} 1 } {\fund {\mathbf G} {\mathbf \Gamma}} {\fund {\mathbf {G/\Gamma}} 1 } {}{\cong}{\cong}{} 
\mspace{70mu} 
\pushout {\dfund {G'} {\Gamma'}} {\dfund {G'/\Gamma'} 1 } {\fund {\mathbf G'} {\mathbf \Gamma'}} {\fund {\mathbf {G'/\Gamma'}} 1 } {}{\cong}{\cong}{} 
\end{equation*}  
where the vertical arrows are isomorphisms and the horizontal arrows are induced by the quotient maps. 
Composing with (\ref{4}) we obtain a commutative diagram in the category of groupoids
\begin{equation*}
\pushout{\dfund{G}{ \Gamma}}{\dfund{G'}{ \Gamma' }}{\dfund{{G/\Gamma}}1}{\dfund{G'/\Gamma'} 1}{\theta}{}{}{\lambda}
\end{equation*}
Moreover, since all the diagrams were induced by the isomorphism $\psi$, it is easy to verify for each $c \in \Gamma$ and $[a] \in  \fund{G}{\Gamma}$ we have $\theta (c \cdot [a]) = \ov{\psi}(c) \cdot \theta ([a])$, where $\ov \psi \colon \Gamma \to \Gamma'$ is defined by $\ov{\psi}(c)G^{00}= \psi(cG^{00})$. 

By Lemma \ref{abelian-homotopy} there is a definable homotopy equivalence $f\colon G/\Gamma \to G'/{\Gamma'}$ with $\pi_1(f) = \lambda$ and $f(1) = 1$. By Lemma \ref{lifting-exists} there is a definable continuous map $f^G \colon G \to G$ with $f^G(1) = 1$ making the following diagram commute.
\begin{equation*}
\pushout{G}{G'}{G/\Gamma}{G'/{\Gamma'}}{f^G}{p}{p'}{f}
\end{equation*}
It remains to show that $f^G$ is a $\Gamma$-equivariant definable homotopy equivalence. 
The equation ${f^G}(c\;\cdot)={f^G}(c){f^G}(\cdot)$ for $c\in\Gamma$ holds because both maps coincide with the unique lifting of $f\circ p$ mapping $1\in G$ to $f^G (c)\in\Gamma'$. Now let $f' \colon G'/\Gamma' \to G/\Gamma$ be a homotopy inverse of $f$. Define ${f^G}'$ to be the unique lifting of
$f'\circ p'$ at $1 \in G$. Then as above ${{f^G}'}(c'y)={{f^G}'}(c'){{f^G}'}(y)$ for
any $c'\in\Gamma'$ and $y\in G'$.
Let $h:I\times G/\Gamma\to G/\Gamma$ be a definable homotopy between the identity $h_0$ on $G/\Gamma$
and $f'\circ f=h_1$, and let $h':I\times G'/{\Gamma'}\to
G'/{\Gamma'}$ be a definable homotopy
between the identity $h'_0$ on $G'/\Gamma'$ and $f\circ f'=h'_1$. We may assume
$h_t(1)=1$ for all $t$ (otherwise use $(t,x)\to
(h_t(1))^{-1}h_t(x)$ instead of $h$) and the same for $h'$.
Finally, define
$\w h:I\times G\to G$ as the unique lifting of
$h\circ({\rm Id}\times p) \colon I \times G \to G/\Gamma$ to $G$
and $\w h':I\times G'\to G'$ as
the unique lifting of $h' \circ({\rm Id}\times p) \colon I \times G' \to G'/\Gamma'$ to $G'$.
By uniqueness of liftings, $\w h$ is a definable homotopy between the
identity and $f^G\circ f^{G'}$. Similarly
$\w h'$ is a definable homotopy between the identity and $f^{G'}\circ f^G$.
Moreover $\w h$ and $\w h'$ are
constant on $I\times \Gamma$ and $I\times \Gamma'$
since $h$ and $h'$ are constant on $I\times\{e\}$.
The equations $\w h_t(cx)=c\w h_t(x)$ and $\w
h'_t(c'x')=c'\w h'_t(x')$, where $c\in \Gamma$ and $c'\in \Gamma'$, follow by uniqueness of liftings.
\ep

\section{Almost direct products} 

Given a group $G$ and two subgroups $A$ and $B$ of $G$, we recall that 
$G$ is the {\bf almost direct product} of $A$ and $B$ if $G=AB$ and the map 
$m \colon A\times B \to G$, $(x,y)\mapsto xy$, is a surjective group homomorphism with a finite kernel. This implies that $ab=ba$ for all $a\in A$ and $b\in B$, and that $\Gamma:= A \cap B$ is a finite central subgroup of $G$. In this situation we write $G = A\times_\Gamma B$. Note that the kernel of $m \colon A\times B \to A\times_\Gamma B$ is $\Gamma^\Delta := \{(c,c^{-1}) \mid c\in \Gamma\}$. 

Every definably compact definably connected group is an almost direct product of a definably connected abelian subgroup and a semisimple definable subgroup. More precisely we have: 

\begin{fact} \label{almost-direct} 
Let $G$ be a definably compact definably connected group. Let $Z^0(G)$ be the definable identity component of the center $Z(G)$ of $G$. By
\cite{HrPePi:08b} the commutator subgroup $[G,G]$ is definable and semisimple, and $G$ is
an almost direct product of $Z^0(G)$ and $[G,G]$. The corresponding statement holds in the category of compact connected Lie groups. 
\end{fact}

\begin{lemma}\label{glue-lemma}
Consider two almost direct products of definable groups $G = A\times_\Gamma B$ and $G' = A'\times_{\Gamma'} B'$. Suppose that there are:
\begin{itemize}
\item an isomorphism $f^\Gamma \colon \Gamma \to \Gamma'$,
\item a $\Gamma$-equivariant definable homotopy equivalence $f^A:A\to A'$, 
\item a $\Gamma$-equivariant definable homotopy equivalence $f^B:B\to B'$, 
\end{itemize}
satisfying $f^A \rest \Gamma = f^B \rest \Gamma = f^\Gamma$.  Then there is a $\Gamma$-equivariant definable homotopy equivalence $f^G:G\to G'$ such that $f^G(ab) = f^A(a)f^B(b)$ for all $a\in A$ and $b\in B$. In particular $f^G\rest A = f^A$ and $f^G \rest B = f^B$. 
\end{lemma}
\begin{proof}
By definition of almost direct product there is a (unique) well defined map $f^G\colon G \to G'$ satisfying $f^G(ab)=f^A(a)f^B(b)$ for $a\in A$ and $b\in B$.
Moreover $f^G$ is continuous since the multiplication $m\colon A\times B \to G$ is a definable covering map (hence locally $f^G = (f^A \otimes f^B) \circ m^{-1}$). Let $f'^A$ and $f'^B$ be homotopy inverses for $f^A$ and $f^B$ satisfying
the conditions of definition \ref{equiva-riant-lence} and let $f'^G:G'\to G$ be defined symmetrically. 
We claim that $f^G$ is a definable homotopy equivalence with homotopy
inverse $f'^G$. In fact, let $h^A:I\times A\to A$ be a definable homotopy between
$f'^A\circ f^A$ and the identity satisfying the conditions of Definition
\ref{equiva-riant-lence},
and let $h^B:I\times B\to B$ be the same for $f'^B\circ f^B$. Define
\begin{equation} \label{bla} h^G_t(ab)=h^A_t(a)h^B_t(b) \end{equation}
for $a\in A$ and $b\in B$. The fact that $h^G_t$ is well defined follows by the conditions in Definition \ref{equiva-riant-lence} and the definition of almost direct product.  
A definable homotopy between $f^G\circ f'^G$
and the identity can be defined symmetrically. The lemma is thus proved.
\end{proof}

\begin{fact} (\cite{Co:09}) \label{conversano} Let $A,B$ be type-definable subgroups of a definable group $G$, $A$ normal in $G$. Then $AB=\{ab \mid a \in A, b\in B\}$ is a type-definable subgroup of $G$ and $(AB)^{00} = A^{00}B^{00}$. \end{fact}

\bl \label{direct} Let $G = A\times_\Gamma B$ be an almost direct product of definable groups. Let $\tau\colon G \to \lie G$ be the natural map. Then $\lie G = \tau(A) \times_{\tau \Gamma} \tau(B)$. \el 
\bp 
Consider the homomorphism $m: \tau(A)\times \tau(B)\to \lie G$ sending $(aG^{00}, bG^{00})$ to $abG^{00}$. Since $G^{00} = A^{00}B^{00}$ (Fact \ref{conversano}), if $abG^{00}$ is the identity of $\lie G$ we have $a a' = b^{-1} b'$ for some $a'\in A^{00}$ and $b' \in B^{00}$. But $A\cap B = \Gamma$, so there is $c\in \Gamma$ such that $a a' = b^{-1} b'= c$. 
It follows that $aG^{00} = cG^{00}$ and $bG^{00} = c^{-1}G^{00}$. We have thus proved that the kernel of $m$ is the finite subgroup $\tau(\Gamma)^\Delta := \{(cG^{00}, c^{-1}G^{00})\mid c \in \Gamma\}$. 
\ep 

\br 
Let $G$ be a definably compact group and let $A$ be a definable subgroup of $G$. 
Let $\tau \colon G \to \lie G$ be the natural map. By \cite{HrPePi:08} we have
$A\cap G^{00} = A^{00}$ (see \cite{Be:07} for the non-abelian case). Therefore \[\tau(A) = AG^{00}/G^{00} \cong A/A^{00}\] via the natural homomorphism sending $aA^{00}\in \lie A$ into $aG^{00} \in AG^{00}/G^{00}$. 
\er

\bl \label{restriction-lemma} Let $G$ be a definably compact definably connected group. Write $G = Z^0(G)\times_\Gamma [G,G]$. Let $\tau \colon G \to \lie G$ be the natural map. Then $\tau(Z^0(G)) = Z^0(\lie G)$ and $\tau ([G,G]) = [\lie G, \lie G]$. So 
\begin{eqnarray*}
\lie G  & = & \tau (Z^0(G)) \; \times_{\tau \Gamma} \; \tau([G,G]) \\
       & = & Z^0(\lie G) \; \times_{\tau \Gamma} \; [\lie G, \lie G] 
\end{eqnarray*}
\el
\bp 
By Fact \ref{almost-direct} $\dim(G) = \dim(Z(G))+\dim([G,G])$ and similarly for $\lie G$. 
By \cite{HrPePi:08} the dimension of $A$ as a definable group equals the dimension of $\lie A$ as a Lie group. So $\tau$ preserves dimensions. The equality $\tau ([G,G]) = [\lie G, \lie G]$ is clear. The inclusion $\tau(Z^0(G)) \subset Z^0(\lie G)$ is also clear (using the fact the image under $\tau$ of a definably connected set is connected). The result follows by counting dimensions. 
\ep 

\bt \label{full} Let $G$ and $G'$ be definably compact definably connected groups. 
Suppose that there is a Lie isomorphism $\psi \colon \lie {G} \to \lie
{G'}$. Then there is a definable homotopy equivalence $f:G\to G'$. Moreover given a finite central subgroup $\Gamma$ of $G$ we can choose $f$ to be $\Gamma$-equivariant, or even $\Gamma'$-equivariant where $\Gamma' = \Gamma [G,G]$ (so in particular $f\rest {[G,G]}$ is an isomorphism onto $[G',G']$). 
If the o-minimal structure is sufficiently saturated we can ensure that $\tau^{G'} \circ f \rest {\Gamma'} = \psi\circ\tau^G \rest {\Gamma'}$ where $\tau_G \colon G \to \lie G$ and $\tau^{G'}\colon G' \to \lie {G'}$ are the projections. 
\et  
\bp We can write $G = Z^0(G) \times_{\Gamma_0} [G,G]$ and we can assume that $\Gamma \supset \Gamma_0$. 
Let $\Gamma_1 = \Gamma [G,G] \cap Z^0(G)$ and note that $\Gamma_1[G,G] = \Gamma[G,G]$. By Theorem \ref{abelian-case} (and lemma \ref{restriction-lemma}) there is 
a definable $\Gamma_1$-equivariant homotopy equivalence $f^Z \colon Z^0(G) \to Z^0(G')$
such that $\tau\circ f^Z\rest {\Gamma}=\psi\circ\tau\rest{\Gamma}$. By Theorem \ref{semisimple-case} (and Lemma \ref{restriction-lemma}) there is a definable isomorphism
$f^{[G,G]} \colon [G,G]\to[G',G']$. In particular both $f^Z$ and $f^{[G,G]}$ are $\Gamma_0$-equivariant definable homotopy equivalences. So by lemma \ref{glue-lemma} there is a $\Gamma_0$-equivariant definable homotopy equivalence $f^G \colon G\to G'$ such that $f^G(ab) = f^Z(a)f^{[G,G]}(b)$ for all $a\in Z^0(G)$ and $b\in [G,G]$. 
By construction, using Equation (\ref{bla}) in Lemma \ref{glue-lemma}, $f^G$ is in fact a $\Gamma_1[G,G]$-equivariant definable homotopy equivalence. If the o-minimal structure is sufficiently saturated (and we choose $f^{[G,G]}$ as in Theorem \ref{semisimple-case}) we obtain $\tau^{G'} \circ f \rest {\Gamma'} = \psi\circ\tau^G \rest {\Gamma'}$. 
\ep 

Another application of Fact \ref{almost-direct} and Lemma \ref{glue-lemma} is the following. 

\bc \label{sa} For each definably compact definably connected group $G$, there is a semialgebraic group over $\R^{\rm{alg}}$, which is definably homotopy equivalent to $G$ and has the same associated Lie group. \ec 
\bp We can write $G = Z^0(G)\times_\Gamma [G,G]$. Let $d = \dim Z^0(G) = \dim Z(G)$. Let $T = SO(2,M)$. Then $T^d$ is semialgebraic over $\R^{\rm{alg}}$ and has the same associated Lie group as $Z^0(G)$. So there is a definable homotopy equivalence $f\colon Z^0(G)\to T^d$ which is $\Gamma$-equivariant. Now let $S$ be a semialgebraic group over $\R^{\rm{alg}}$ definably isomorphic to $[G,G]$. So in particular there is a $\Gamma$-equivariant definable homotopy equivalence $g\colon [G,G]\to S$. We can assume that $f$ and $g$ are the identity on $\Gamma$. By Lemma \ref{glue-lemma} $G$ is definably homotopy equivalent to $G':= T^d \times_\Gamma S$. Note that $G'$ is semialgebraic over $\R^{\rm{alg}}$ and has the same associated Lie group as $G$ by Lemma \ref{direct}. 
\ep

The following corollary was proved in the abelian case in \cite{BeMaOt:09}. Granted the above Corollary the same proof extends to the general case. 

\bc For each definably compact definably connected group $G$, $$H^{\rm{def}}_*(G;\Z) \cong H_*(\lie G; \Z)$$ where $H^{\rm{def}}$ is the o-minimal homology functor. \ec

\end{document}